\documentclass[11pt,reqno]{amsart}

\usepackage{amssymb,dirtytalk,comment,tikz-cd,tikz,color,cite,enumerate,setspace}
\usepackage[unicode=true]
 {hyperref}
\hypersetup{
colorlinks=true,
urlcolor=black,
citecolor=blue,
linkcolor=blue,
}
\usepackage{xltabular}
\usepackage{pdflscape}

\newtheorem{theorem}{Theorem}[section]
\newtheorem{lem}[theorem]{Lemma}
\newtheorem{prop}[theorem]{Proposition}

\theoremstyle{definition}
\newtheorem{definition}[theorem]{Definition}

\newtheorem*{notation*}{Notation}
\theoremstyle{remark}
\newtheorem{remark}[theorem]{Remark}

\numberwithin{equation}{section}

\newcolumntype{L}[1]{>{\raggedright\arraybackslash}p{#1}}
\newcolumntype{C}[1]{>{\centering\arraybackslash}p{#1}}
\newcolumntype{R}[1]{>{\raggedleft\arraybackslash}p{#1}}

\allowdisplaybreaks
\begin{document}

\title[Left series is not generated by right-normed star products]{Skew braces whose left series is not generated by right-normed star products}

\author{Cindy (Sin Yi) Tsang}
\address{Department of Mathematics, Ochanomizu University, 2-1-1 Otsuka, Bunkyo-ku, Tokyo, Japan}
\email{tsang.sin.yi@ocha.ac.jp}
\urladdr{http://sites.google.com/site/cindysinyitsang/}



\begin{abstract} We construct explicit examples of skew braces $A$ for which the third term $A^3$ in the left series is not generated by the right-normed star products $a_1*(a_2*a_3)$ of length $3$ as a sub-skew brace.\end{abstract}

\maketitle

\vspace{-5mm}

\section{Introduction} 

Let $A = (A,+,\circ)$ be a \textit{skew brace}. This means that $(A,+)$ and $(A,\circ)$ are groups, not necessarily abelian, such that
\begin{equation}\label{eqn:relation}
 a \circ (b+c) = a\circ b -a + a\circ c
\end{equation}
holds for all $a,b,c\in A$. This relation, which resembles the left-distributivity in a ring, implies that $(A,+)$ and $(A,\circ)$ share the same identity element $0$. A \textit{sub-skew brace} is a subset that is a subgroup of both $(A,+)$ and $(A,\circ)$.

\smallskip

For any elements $a,b\in A$, their \textit{star product} is defined by
\[ a*b = - a + a\circ b - b.\]
Clearly $a*b=0$ if and only if $a\circ b = a+b$. Thus, the operation $*$ measures the difference between $(A,+)$ and $(A,\circ)$. Since the commutator of a group measures the difference between the group and its opposite group, in some sense, the star product is an analog of the commutator.

\smallskip

For any subsets $X,Y\subseteq A$, their \textit{star product} is defined by
\[ \langle x*y : x\in X,\, y\in X\rangle_+\]
where $\langle - \rangle_+$ denotes subgroup generation in $(A,+)$. Then $A*A$ is, in some sense, an analog of the commutator subgroup. Let us note that $A*A$ is a sub-skew brace (in fact an \textit{ideal}) of $A$; see \cite[Proposition 2.1]{series}.

\smallskip

In contrast to the group commutator, it is not true that $B*C= C*B$ in general for sub-skew braces $B,\, C$. As a consequence, there are at least two different ways to define analogs of the lower central series, as follows.

\begin{definition} The \textit{left series} of $A$ is defined by
\[ A^1 = A,\quad A^{(k+1)} = A*A^k\mbox{ for }k\geq 1.\]
The \textit{right series} of $A$ is defined by
\[ A^{(1)} = A,\quad A^{(k+1)} = A^{(k)} * A\mbox{ for }k\geq 1. \]
Let us remark that $A^k$ and $A^{(k)}$ are sub-skew braces (in fact \textit{left ideals} and \textit{ideals}, respectively) of $A$ for all $k\geq 1$; see \cite[Propositions 2.2 and 2.1]{series}.
\end{definition}

We mention in passing that there are other analogs of the lower central series in skew braces; see \cite[Definitions 2.10 and 2.13]{lower}.

\smallskip

It is well known that for a group, the $k$th term in the lower central series is generated by the right-normed commutators of length $k$, and similarly by the left-normed commutators of length $k$. It is natural to ask whether this admits an analog when we consider the left and right series. This question was investigated in \cite{Arai}, where the following definition was introduced:

\begin{definition} The \textit{verbal left series} of $A$ is defined by
\[ A^{\underline{k}} = \langle a_1 *( a_2*(\cdots * (a_{k-1}*a_k)\cdots )): a_1,\dots,a_k\in A \rangle\mbox{ for }k\geq 1.\]
The \textit{verbal right series} of $A$ is defined by
\[ A^{(\underline{k})} =\langle ((\cdots (a_1*a_2)*\cdots )*a_{k-1})*a_k : a_1,\dots,a_k\in A\rangle\mbox{ for }k\geq 1.\]
Here $\langle -\rangle$ denotes sub-skew brace generation in $A$.
\end{definition}

Obviously, we have 
\begin{align*} A &= A^1 = A^{\underline{1}} = A^{(1)} = A^{(\underline{1})},\\
A*A &= A^2 = A^{\underline{2}} = A^{(2)} = A^{(\underline{2})}.
\end{align*}
For $k\geq 3$, we always have the inclusions
\[ A^{\underline{k}}\subseteq A^{k},\quad A^{(\underline{k})} \subseteq A^{(k)}, \]
but it is unclear whether they are equalities in general. It was shown in \cite[Propositions 3.5 and 3.6]{Arai}, respectively, that $A^{\underline{k}} = A^{k}$ and $A^{(\underline{k})} = A^{(k)}$ for all $k\geq 1$, when some extra assumptions are imposed.

\smallskip

In this paper, we focus on the left series and verbal left series. The goal is to exhibit explicit examples of skew braces $A$ such that $A^{\underline{3}}\subsetneq A^{3}$. We will first describe a general method of constructing such examples in Section \ref{sec:construction}. We then construct an infinite family of examples in Section \ref{sec:1}.

\section{A general method of construction}\label{sec:construction}

For any group $B=(B,\bullet)$ and $\varphi \in \mathrm{Aut}(B)$, let $A = (B\times \langle \varphi\rangle, + ,\circ)$ with
\begin{align*}
(b_1,\varphi^{i_1}) + (b_2,\varphi^{i_2}) & = (b_1\bullet b_2,\varphi^{i_1+i_2}),\\
(b_1,\varphi^{i_1}) \circ (b_2,\varphi^{i_2})  & = (b_1\bullet \varphi^{i_1}(b_2),\varphi^{i_1+i_2}).
\end{align*}
Hence $(A,+)$ is the direct product $B\times \langle \varphi\rangle$, and $(A,\circ)$ is the natural semi-direct product $B\rtimes \langle \varphi\rangle$. It is easy to verify \eqref{eqn:relation}, so that $A$ is a skew brace. In fact, this method of construction is well-known; see \cite[Example 1.4]{skew}.

\smallskip

For any $b\in B$ and $i\in \mathbb{Z}$, define the map
\[ \varphi_i : B\rightarrow B;\quad \varphi_i(b)= \varphi^i(b)\bullet b^{-1}.\]
For any $b_1,b_2\in B$ and $i_1,i_2\in \mathbb{Z}$, we then see that
\begin{align*}
(b_1,\varphi^{i_1}) *(b_2,\varphi^{i_2}) 
& = (\varphi^{i_1}(b_2)\bullet b_2^{-1},\mathrm{id}) = (\varphi_{i_1}(b_2),\mathrm{id}).
\end{align*}
We now deduce that
\begin{align*}
\{ a_1*a_2: a_1,a_2\in A\} & = \{ (\varphi_j(b),\mathrm{id}): b\in B,\, j\in\mathbb{Z}\},\\
\{a_1*(a_2*a_3) : a_1,a_2,a_3\in A\} & = \{ (\varphi_i(\varphi_j(b)),\mathrm{id}):b\in B,\, i,j\in\mathbb{Z}\}.
\end{align*}
Since both $+$ and $\circ$ coincide with $\bullet$ on $B\times \{\mathrm{id}\}$, it follows that
\begin{align}\label{eqn:A*A}
A*A & = \langle \varphi_j(b): b\in B,\, j\in\mathbb{Z}\rangle_\bullet\times \{\mathrm{id}\},\\\notag
A^3 & = \langle \varphi_i(b') : b' \in \langle \varphi_j(b): b\in B,\, j\in\mathbb{Z}\rangle_\bullet,\, i\in \mathbb{Z}\rangle_\bullet \times \{\mathrm{id}\},\\\notag
A^{\underline{3}} & = \langle\varphi_i(\varphi_j(b)): b \in B,\, i,j\in\mathbb{Z}\rangle_\bullet \times \{\mathrm{id}\},
\end{align}
where $\langle-\rangle_\bullet$ denotes subgroup generation in $B$. We then obtain:

\begin{prop}\label{prop:general}Suppose that
\begin{align}\notag
&\langle \varphi_i(\varphi_j(b)) : b\in B,\, i,j\in\mathbb{Z}\rangle_\bullet\\\label{eqn:condition}
&\hspace{1.5cm}\subsetneq \langle \varphi_i(b') : b'\in \langle \varphi_j(b) : b \in B,\, j\in \mathbb{Z}\rangle_\bullet,\, i\in\mathbb{Z} \rangle_\bullet.
 \end{align}
Then the skew brace $A$ constructed above satisfies $A^{\underline{3}}\subsetneq A^3$.
\end{prop}

\begin{remark} For any $b,b'\in B$ and $i\in \mathbb{Z}$, we have
\[ (b',\mathrm{id}) * (b,\varphi^i)  = (\varphi_0(b),\mathrm{id}) = (1,\mathrm{id}).\]
Together with \eqref{eqn:A*A}, this yields $A^{(\underline{3})} = A^{(3)} = 0$. Thus, the above construction cannot be used to create skew braces $A$ for which $A^{(\underline{3})}\subsetneq A^{(3)}$.
\end{remark}

We now give examples of $B = (B,\bullet)$ and $\varphi\in \mathrm{Aut}(B)$ satisfying \eqref{eqn:condition}.
\section{An infinite family of examples}\label{sec:1}

We will apply Proposition \ref{prop:general} to prove:

\begin{theorem}\label{thm1} There exist infinitely many skew braces $A$ of order a power of $2$ such that $A^{\underline{3}}\subsetneq A^3$.
\end{theorem}

Let $n\geq 4$ and $n\geq \ell \geq 0$ be integers. Let $B = (B,\bullet)$ be the set $(\mathbb{Z}/2^n\mathbb{Z})^3$ equipped with the operation $\bullet$ defined by
\[ \begin{pmatrix}x_1\\x_2\\x_3\end{pmatrix} \bullet \begin{pmatrix} y_1\\y_2\\y_3\end{pmatrix}
 \begin{pmatrix} x_1+y_1\\x_2+y_2\\x_3+y_3+2^\ell x_1y_2 \end{pmatrix}.\]
 It is straightforward to check that $B$ is a group. The identity element is the zero vector, and the inverse of an arbitrary element is given by
 \[ \begin{pmatrix}x_1\\x_2\\x_3\end{pmatrix}^{-1} = \begin{pmatrix} -x_1\\-x_2 \\ -x_3 + 2^\ell x_1x_2\end{pmatrix}. \]
Let $s,t$ be any odd integers. Again, it is straightforward to check that
\[ \varphi: \begin{pmatrix}x_1\\x_2\\x_3\end{pmatrix} \mapsto \begin{pmatrix} sx_1\\tx_2\\stx_3\end{pmatrix}\]
defines an automorphism of $B$. For each $i\in \mathbb{Z}$, we have
\[ \varphi_i : \begin{pmatrix} x_1\\x_2\\x_3\end{pmatrix} \mapsto \begin{pmatrix} s^ix_1\\t^ix_2\\s^it^ix_3\end{pmatrix}\bullet \begin{pmatrix}x_1\\x_2\\x_3\end{pmatrix}^{-1} = \begin{pmatrix}
(s^i-1)x_1\\ (t^i-1)x_2\\ (s^it^i-1)x_3 - 2^\ell (s^i-1)x_1x_2\end{pmatrix}.\]
We will assume that $\ell$ and $s,t$ satisfy all of the following:
\begin{enumerate}[(1)]
\item $(s-1)^2(t-1)^2 \equiv 0\pmod{2^{n-\ell}}$;
\item $(st-1)^2 \equiv0 \pmod{2^n}$;
\item $ (s^j-1)\left\{(s^it^i -1) + (s^i-1)(t^j-1)\right\} \equiv 0\pmod{2^{n-\ell}}$ for all $i,j\in\mathbb{Z}$;
\item $(s-1)(st-1)\not\equiv 0\pmod{2^{n-\ell}}$.
\end{enumerate}
Moreover, for simplicity let us put
\[ U = \left\{ \begin{pmatrix}x_1\\x_2\\0\end{pmatrix} \in B : \begin{array}{c}
x_1\equiv 0 \hspace{-2mm}\pmod{2^{2v_2(s-1)}}\\
x_2\equiv 0 \hspace{-2mm}\pmod{2^{2v_2(t-1)}}
\end{array}\right\},\]
where $v_2(-)$ denotes $2$-adic valuation.

\begin{lem}\label{lem:U} The above subset $U$ is a subgroup of $B$.
\end{lem}
\begin{proof} For any $x_1,x_2,y_1,y_2\in \mathbb{Z}/2^n\mathbb{Z}$, observe that
\begin{equation}\label{eqn:difference} \begin{pmatrix}
x_1\\x_2\\0\end{pmatrix}\bullet \begin{pmatrix} y_1\\y_2\\0\end{pmatrix}^{-1} = 
\begin{pmatrix} x_1-y_1\\ x_2-y_2\\
-2^\ell(x_1-y_1)y_2
\end{pmatrix}.\end{equation}
For any $(x_1,x_2,0)^T,(y_1,y_2,0)^T\in U$, we then clearly have
\[ x_1-y_1\equiv 0\hspace{-2mm}\pmod{2^{2v_2(s-1)}},\quad x_2-y_2\equiv 0\hspace{-2mm}\pmod{2^{2v_2(t-1)}},\]
\[ -2^\ell(x_1-y_1)y_2 \equiv 0 \hspace{-2mm}\pmod{2^{\ell + 2v_2(s-1)+2v_2(t-1)}},\]
and the last modulus divides $2^n$ by condition (1). Thus \eqref{eqn:difference} lies in $U$.\end{proof}

Note that the skew brace $A = (B\times \langle\varphi\rangle,+,\circ)$ constructed in Section \ref{sec:construction} is of order $2^{3n}|\varphi|$ here, which is a power of $2$. We show that \eqref{eqn:condition} is satisfied.

\begin{lem}\label{lem:U1}We have $\langle \varphi_i(\varphi_j(b)) : b\in B,\, i,j\in\mathbb{Z}\rangle_\bullet \subseteq U$.
\end{lem}
\begin{proof} By Lemma \ref{lem:U}, it is enough to show that the generators all lie in $U$. For any $b = (x_1,x_2,x_3)^T\in B$ and $i,j\in \mathbb{Z}$, observe that
\[ \varphi_i(\varphi_j(b))
= \begin{pmatrix}
(s^i-1)(s^j-1)x_1\\ (t^i-1)(t^j-1)x_2\\
\clubsuit
\end{pmatrix},
\]
where the third coordinate $\clubsuit$ is given by
\begin{align*}
&(s^it^i-1)\left\{(s^jt^j-1)x_3 - 2^\ell(s^j-1)x_1x_2\right\}\\
&\hspace{2cm} - 2^\ell(s^i-1)\left\{(s^j-1)x_1\right\}\left\{(t^j-1)x_2\right\}.
\end{align*}
We may assume that $i,j\geq 1$. For the first and second coordinates, respectively, we clearly have
\begin{align*}
(s^i-1)(s^j-1)x_1&\equiv (s-1)^2 (\mbox{some integer})\equiv0&&\hspace{-11mm}\pmod{2^{2v_2(s-1)}},\\
(t^i-1)(t^j-1)x_2&\equiv (t-1)^2(\mbox{some integer})\equiv0&&\hspace{-11mm}\pmod{2^{2v_2(t-1)}}.
\end{align*}
For the third coordinate, similarly the coefficient of $x_3$ is 
\[ (s^it^i-1)(s^jt^j-1) \equiv (st-1)^2(\mbox{some integer}) \equiv 0 \hspace{-2mm}\pmod{2^n}\]
by condition (2), while the coefficient of $x_1x_2$ is 
\[ -2^\ell(s^j-1)\left\{ (s^it^i-1) + (s^i-1)(t^j-1)\right\} \equiv 0\hspace{-2mm}\pmod{2^n}\]
by condition (3). Thus, we have $\varphi_i(\varphi_j(b))\in U$, as desired. \end{proof}

\begin{lem}\label{lem:U2} We have $\langle \varphi_i(b') : b'\in \langle \varphi_j(b) : b \in B,\, j\in \mathbb{Z}\rangle_\bullet ,\, i\in \mathbb{Z}\rangle_\bullet\not\subseteq U$.
\end{lem}
\begin{proof} The subgroup $\langle \varphi_j(b) : b \in B,\, j \in \mathbb{Z}\rangle_\bullet$ contains
\[b'=\varphi_1\begin{pmatrix}1\\1\\0\end{pmatrix}\bullet \varphi_1\begin{pmatrix}0\\-1\\0\end{pmatrix}\bullet \varphi_1\begin{pmatrix}-1\\0\\0\end{pmatrix},\]
and we compute that
\[b' = \begin{pmatrix} s-1\\t-1\\-2^\ell(s-1)\end{pmatrix}\bullet \begin{pmatrix} 0 \\ 1-t \\ 0 \end{pmatrix}\bullet\begin{pmatrix}1-s\\0\\0\end{pmatrix}= 
\begin{pmatrix}
0\\0\\ -2^\ell(s-1)t
\end{pmatrix}.\]
Hence, the subgroup $\langle \varphi_i(b') : b'\in \langle \varphi_j(b) : b \in B,\, j\in \mathbb{Z}\rangle_\bullet ,\, i\in \mathbb{Z}\rangle_\bullet$ contains
\[ \varphi_1(b')= \begin{pmatrix}
0\\0\\ -2^\ell(s-1)(st-1)t\end{pmatrix}.
\]
But this is not an element of $U$ by condition (4).
\end{proof}

Lemmas \ref{lem:U1} and \ref{lem:U2} imply that the condition \eqref{eqn:condition} in Proposition \ref{prop:general} is satisfied. Theorem \ref{thm1} now follows from Proposition \ref{prop:n}, which shows that we can choose $\ell$ and $s,t$ to satisfy conditions (1) $\sim$ (4) when $n\neq 5,6,9$.

\smallskip

Before we proceed, let us consider the $m\in\mathbb{N}$ that satisfy
\[ 3m+1\leq n\leq 4m,\quad\mbox{or equivalently},\quad\frac{n}{4} \leq m \leq \frac{n-1}{3}.\]
For $4\leq n\leq 15$, we have the following, where $\times$ means ``no solution".
\[ \begin{tabular}{|C{5mm}||C{5mm}|C{5mm}|C{5mm}|C{5mm}|C{5mm}|C{5mm}|C{5mm}|C{5mm}|C{5mm}|C{5mm}|C{5mm}|C{5mm}|}
\hline
$n$ & $4$ & $5$ & $6$ & $7$ & $8$ & $9$ & $10$ & $11$ & $12$ & $13$ & $14$ & $15$ \\\hline
$m$ & $1$ & $\times$ & $\times$ & $2$ & $2$ & $\times$ & $3$ & $3$ & $3$ & $4$ &$4$  & $4$ \\\hline
\end{tabular}\] 
For $n\geq 16$, there is always a solution $m\in\mathbb{N}$ because
\[ \frac{n-1}{3} - \frac{n}{4} = \frac{n-4}{12} \geq 1.\]
This is why we exclude the values $n=5,6,9$ in  Proposition \ref{prop:n} below.

\begin{prop}\label{prop:n} For $n\geq 4$ with $n\neq 5,6,9$, let $m\in \mathbb{N}$ be such that 
\[3m+1\leq n\leq 4m.\]
Then conditions $(1)$ $\sim$ $(4)$ hold for
\[ \ell = n - (3m+1),\quad s = 1+2^m,\quad t = 1-2^m.\]
\end{prop}

\begin{proof}For conditions (1), (2), and (4), they may be rewritten as follows:
\begin{enumerate}[(1)]
\item[(1)] $(2^m)^2(-2^m)^2 \equiv 0\pmod{2^{3m+1}}$;
\item[(2)] $(-2^{2m})^2 \equiv0 \pmod{2^n}$;
\item[(4)] $(2^m)(-2^{2m})\not\equiv 0\pmod{2^{3m+1}}$.
\end{enumerate}
All of them are obvious. For condition (3), let $i,j\in \mathbb{Z}$, and consider
\[ \spadesuit_{i,j} = (s^j-1)\left\{(s^it^i -1) + (s^i-1)(t^j-1)\right\}.\]
We need to show that $\spadesuit_{i,j}\equiv 0\pmod{2^{3m+1}}$. We may assume that $i,j\geq 1$.  

\smallskip

First, suppose that $n=4$. Then $m=1$ and $(s,t)=(3,-1)$. We have
\[ \spadesuit_{i,j}\equiv \begin{cases}
(3^j-1)(( -3)^i-1)&\hspace{-5mm}\pmod{2^4}\mbox{ when $j$ is even},\\
(3^j-1)(3^i-1)(-1)^j &\hspace{-5mm}\pmod{2^4}\mbox{ when $i$ is even},\\
-(3^j-1)(3^{i+1}-1)&\hspace{-5mm}\pmod{2^4}\mbox{ when $i,j$ are both odd}.
\end{cases}\]
Since $(-3)^k,3^k\equiv 1\pmod{2}$ for all $k$, and $3^k\equiv1\pmod{8}$ for $k$ even, in all three cases we see that $\spadesuit_{i,j}\equiv 0\pmod{2^{4}}$.
 
\smallskip

Next, suppose that $n\geq 7$. Then $m\geq 2$ and so $s,t\equiv1\pmod{4}$. We have 
\[v_2(\spadesuit_{i,j}) = v_2(s^j-1) + v_2((s^it^i-1) + (s^i-1)(t^j-1)).\]
For any $r\equiv1\pmod{4}$ and $k\geq 1$, the lifting-the-exponent lemma yields
\[ v_2(r^k-1) = v_2(r-1) + v_2(k).\]
Applying this to $r=s,t,st$, we obtain
\begin{align*}
v_2(s^j-1) & = m + v_2(j),\\
v_2(s^it^i-1) & = 2m + v_2(i),\\
v_2((s^i-1)(t^j-1)) & = 2m + v_2(i) + v_2(j).
\end{align*}
For any integers $r,r'$, observe that
\[ v_2(r+r') = \begin{cases} \min\{ v_2(r),v_2(r')\} & \mbox{when }v_2(r)\neq v_2(r'), \\
v_2(r) + v \,\ (\exists v\geq 1) & \mbox{when }v_2(r) = v_2(r').
\end{cases}\]
Applying this to $r= s^it^i-1$ and $r' = (s^i-1)(t^j-1)$, we deduce that
\[ v_2(\spadesuit_{i,j})  = \begin{cases}
m + v_2(j) + 2m + v_2(i) &\mbox{when }v_2(j)\geq 1,\\
m + v_2(j) + 2m+v_2(i) + v \,\ (\exists v\geq 1)&\mbox{when }v_2(j)= 0.
\end{cases}\]
In both cases, we get that $\spadesuit_{i,j}\equiv 0\pmod{2^{3m+1}}$.

\smallskip

This completes the proof.
\end{proof}


%
%
%
%
%
%

\section*{Acknowledgements}

The author would like to thank Noriyuki Abe for helpful discussion.

\smallskip

This work is supported by JSPS KAKENHI 24K16891.

\end{document}